\documentclass{rmmcart}

\usepackage{amssymb}
\usepackage{amsmath}
\usepackage{amsthm}
\usepackage{hyperref}
\usepackage{setspace}

\title{On the containment problem for fat points}
\author{Iman Bahmani Jafarloo}
\address{ Dipartimento di Scienze Matematiche, Politecnico di Torino, Italy. \newline
Dipartimento di Matematica, Universit\`a di Torino, Italy.}
\email{iman.bahmanijafarloo@polito.it,\newline
		iman.bahmanijafarloo@unito.it}
\thanks{Corresponding author:  Iman Bahmani Jafarloo}
\author{Giuseppe Zito}
\address{Dipartimento di Matematica e Informatica, Universit\`a  di Catania, Italy.}
\email{giuseppezito@hotmail.it}
\date{July 26, 2018}
\begin{document}
\begin{abstract}
Given an ideal $I$, the containment problem is concerned with finding the values $m$ and $r$ such that the $m$-th symbolic power of $I$ is contained in its $r$-th ordinary power. A central issue related to this is determining the resurgence for ideals $I$ of fat points in projective space. In this paper we obtain complete results for the resurgence of fat point schemes $m_0P_0+m_1P_1+m_2P_2$ in $\mathbb{P}^N$ for any distinct points $P_0,P_1,P_2$, and, when the points $P_0,\ldots,P_n$ are collinear, we extend this result to fat point schemes $m_0P_0+\dots+m_nP_n$. As a by-product of our determining the resurgence for all three points fat point ideals, we give new examples of ideals with symbolic defect zero. In case the points are noncollinear, a three  fat points ideal can be regarded as a monomial ideal, but it is typically not square-free.
\end{abstract}
\maketitle
\section{introduction}
\subsection{Background}
Let us consider the polynomial ring $R=\mathbb{K}[\mathbb{P}^N]=\mathbb{K}[x_0,\ldots,x_N]$, where $\mathbb{K}$ is an algebraically closed field of any characteristic and $N \geq 2$.
In general, if $I$ is a homogeneous ideal of $R$, the $m$-th symbolic power of $I$ is $I^{(m)}=R \cap (\bigcap_{p\in \textrm{ASS}(I)}{(I^mR_{p})})$.
However in this paper we will always deal with ideals of fat points that are ideals of the form $I=\bigcap_i{I(P_i)^{m_i}}$, where $P_i$ are distinct points in $\mathbb{P}^N$, $I(P_i)$ is the ideal of all the forms that vanish at $P_i$ and the multiplicity $m_i$ is a non-negative integer.
For ideals of this type, the $m$-th symbolic power can be simply defined as $I^{(m)}=\bigcap_i{I(P_i)^{mm_i}}$. During the last decades, there has been a lot of interest  comparing powers of ideals with symbolic
powers in various ways; see for example, \cite{Hochster1973}, \cite{Swanson2000}, \cite{Kodiyalam2000}, \cite{Ein2001}, \cite{Hochster2002}, 
and \cite{Li2006}.
It is easy to see that $I^r \subseteq I^{(r)} \subseteq I^{(m)}$ if and only if $r \geq m $.
Furthermore $I^{(m)} \subseteq I^r $ implies $m \geq r$ but the converse is not true in general. Therefore it makes sense to ask the containment question: given an ideal $I$, for which  $m$ and $r$ is the symbolic power $I^{(m)} $ contained in the ordinary power $I^r$?
In \cite{bocci2010resurgence} and \cite{bocci2010comparing}, Bocci and Harbourne introduced and studied an asymptotic quantity, known as the resurgence, whose computation is clearly linked to the containment problem.
\begin{definition}
Given a non-zero proper homogeneous ideal $I$ in $R$, the resurgence of $I$, denoted by $\rho(I)$, is defined as the quantity: $$\rho(I)=\sup\left\{m/r: I^{(m)}\not \subseteq I^r \right\}.$$
\end{definition}
From the results of \cite{Ein2001,Hochster2002}, it follows that $I^{(m)} \subseteq I^r$ whenever $m \geq Nr$. Thus, we can conclude that $\rho(I) \leq N$ for any homogeneous ideal in $R$.
In general, directly computing $\rho(I)$ is quite difficult and it has been determined only in very special cases. For example, it is known that $\rho(I)=1$ when $I$ is generated by a regular sequence \cite{bocci2010comparing}. The resurgence is also known for certain cases of ideals of the following kinds: monomial ideals
\cite[Theorem 4.11]{Geramita2013}, ideals of a projective cone \cite[Proposition 2.5.1]{bocci2010comparing} and ideals of points on a reducible conic in $ \mathbb{P}^2 $ \cite{DENKERT2013120}.

Another situation where the resurgence is known is for certain ideals $I$ defining
zero-dimensional subschemes of projective space. For example, if $\alpha(I) = \textrm{reg}(I)$, where $\textrm{reg}(I)$ is the Castelnuovo-Mumford
regularity of $I$ and $\alpha(I)$ is the degree of a non-zero element of $I$ of least degree, then the resurgence can be completely described in terms of numerical invariants of $ I $ (\cite[Corollary 2.3.7]{bocci2010comparing}
and \cite[Corollary 1.2]{bocci2010resurgence}).
One of these invariants is called the Waldschmidt constant of $I$.
\begin{definition}
Let $ I $ be a non-zero proper homogeneous ideal of $ R $. The Waldschmidt constant of $I$, denoted by $ \widehat{\alpha}(I) $, is defined as: 
$$\widehat{\alpha}(I)=\inf_{m> 0}\{\alpha(I^{(m)})/{m}\}=\lim_{m \to \infty }{\alpha(I^{(m)})/{m}}.$$
\end{definition}
In particular, when $I$ defines a $0$-dimensional subscheme, the authors in \cite[Theorem 1.2]{bocci2010resurgence} proved that $ \frac{\alpha(I)}{\widehat{\alpha}(I)} \leq \rho(I) \leq \frac{\textrm{reg}(I)}{\widehat{\alpha}(I)}$, so $\rho(I)= \frac{\alpha(I)}{\widehat{\alpha}(I)}$ when $\alpha(I)=\textrm{reg}(I)$. 
\subsection{Preliminaries}
Hereafter $ Z $ is a fat point scheme of $ \mathbb{P}^N $.

\begin{definition}
Let $P_1,\ldots,P_n$ be distinct points in $ \mathbb{P}^N $ and $ m_1,\ldots,m_n $ be non-negative integers. The ideal $I=\bigcap_{i=1}^n{I(P_i)^{m_i}}$ defines a subscheme of $ \mathbb{P}^N $ and we will denote it by $Z=\sum_{i=1}^nm_iP_i$ where by definition we set $ I(Z) = I $.
\end{definition}

\begin{remark}
Consider a fat point subscheme $Z=\sum{m_tP_t}$  where all the points $P_t$ lie on a plane  $\Pi$ (hence a ${\mathbb{P}}^2$, unique if and only if the points are not collinear). The subscheme $\Pi\cap Z$ is a fat point subscheme of $\Pi={\mathbb{P}}^2$. We denote the ideal of $\Pi\cap Z$ in $R_\Pi=\mathbb K[\Pi]$ by $I_\Pi(Z)$, or more simply by 
$I_{{\mathbb{P}}^2}(Z)\subseteq R_{{\mathbb{P}}^2}$. Thus $I_{{\mathbb{P}}^2}(Z)=\bigcap{I_{{\mathbb{P}}^2}(P_t)^{m_t}}
$,
and for emphasis we may denote $I(Z)\subseteq R=\mathbb K[{\mathbb{P}}^N]$ by
$I_{{\mathbb{P}}^N}(Z)\subseteq R_{{\mathbb{P}}^N}=\mathbb K[{\mathbb{P}}^N]$.	
\end{remark}

In this paper we compute the resurgence of $I(Z)$ for two classes of fat point subschemes.

In Section \ref{section2}, we study the subscheme $Z=\sum_{i=1}^n{m_iP_i}$ in $\mathbb{P}^N$, where the points $P_i$ are collinear and we determine the resurgence in Theorem \ref{theo}.
\begin{theorem} \label{theo}
	Let $Z=\sum_{i=1}^n{m_iP_i}$ be a fat point scheme, where $P_1,\ldots,P_n$ are distinct collinear points in $\mathbb{P}^N$. Then $ I(Z)^{(m)}=I(Z)^m \textrm{ for all } m \in \mathbb{N}$, thus $\rho(I(Z))=1$.
\end{theorem}
The proof of Theorem \ref{theo} is a direct consequence of Lemma \ref{criterion} and Lemma \ref{claim}. 

The property that $I(Z)^{(m)}=I(Z)^m$, presented in the statement of Theorem \ref{theo}, gives us more information than the exact value of the resurgence.
In particular, if we define, as in \cite{defect}, the $m$-symbolic defect of a homogeneous ideal $I$ of $R$ as the minimal number $\textrm{sdefect}(I,m)$  of generators of the $R$-module $I^{(m)}/I^m$, Theorem \ref{theo} tells us that $\textrm{sdefect}(I(Z),m)=0$ for all $m$, if $Z$ is a fat point scheme whose support consists of collinear points.
Notice that $\textrm{sdefect}(I(Z),m)=0$ for all $m \geq 1$ implies $\rho(I(Z))=1$. It is still not known if there exists a scheme $Z$ with $\rho(I(Z))=1$ and $I(Z)^m \neq I(Z)^{(m)}$ for some $m$, but our results show that any such $Z$ must be supported at more than three points, not all of which can be collinear.

In Section \ref{section3} we start considering the fat point subschemes $Z$ consisting of three noncollinear points, initially focusing on the $\mathbb{P}^2$ case. In particular, we show how  the invariants $\alpha(I_{\mathbb{P}^2}(Z))$ and $\widehat{\alpha}(I_{\mathbb{P}^2}(Z))$ depend on the values assigned to the multiplicities and how to relate the value of the resurgence of $I_{\mathbb{P}^2}(Z)$ to
$\rho(I_{\mathbb{P}^N}(Z))$ (hereafter, we will always use $I(Z)$ to mean $I_{\mathbb{P}^N}(Z)$).

In Section \ref{section4}, we consider  the subscheme $Z=m_0P_0+m_1P_1+m_2P_2$, where the $P_i$'s are noncollinear points in $\mathbb{P}^N$ and $m_0\leq m_1 \leq m_2$ are nonnegative integers. In Theorem \ref{thmpn} we classify fat point ideals
in $\mathbb{P}^N$ supported at three noncollinear points which have $m$-symbolic defect zero for all $m$ (and hence such that $\rho(I(Z))=1$). 
\begin{theorem} \label{thmpn}
Let $P_0, P_1$ and $P_2$ be noncollinear points in $\mathbb{P}^N$ and $m_2 \geq \max(m_0,m_1)  $. Consider the fat point scheme $Z=m_0P_0+m_1P_1+m_2P_2$.
Then $\textrm{sdefect}(I(Z),m)=0$ for all $m\in \mathbb{N}$ if and only if one of the following conditions holds:
\begin{itemize}
\item[\textbf{(a)}]
$m_0+m_1 \leq m_2$;
\item[\textbf{(b)}]
$m_0+m_1> m_2$ and $m_0+m_1+m_2$ is even.
\end{itemize}
\end{theorem}
The proof that Theorem \ref{thmpn}
(a) implies $\textrm{sdefect}(I(Z),m) = 0$ for all $m \geq 1$ is Proposition \ref{cor1}. The proof that Theorem
\ref{thmpn} (b) also implies $\textrm{sdefect}(I(Z),m) = 0$ for all $m \geq 1$ is Proposition \ref{cor7}. To complete
the proof of Theorem \ref{thmpn}, it remains to show that $\textrm{sdefect}(I(Z),m) > 0$ for some $m > 0$
whenever $m_0 + m_1 > m_2$ and $m_0 + m_1 + m_2$ is odd. This follows from Theorem \ref{casec}.

\begin{theorem} \label{casec}
Let $P_0, P_1$ and $P_2$ be noncollinear points in $\mathbb{P}^N$ and $ \max( m_0,m_1) \leq m_2  $. Consider the fat point scheme $Z=m_0P_0+m_1P_1+m_2P_2$. If  $ m_0+m_1 > m_2 $ and $ 	m_0+m_1+m_2  $ is odd, then $$ \rho(I(Z))= \frac{m_0+m_1+m_2+1}{m_0+m_1+m_2}.$$
\end{theorem}
The proof follows at once from Corollary \ref{cor5} and Proposition \ref{prop}.
\section{Fat points on a line in $\mathbb{P}^N$} \label{section2}
Let $ L $ be a line in $ \mathbb{P}^N  $ and let $P_1,\ldots,P_n$ be distinct points which lie on $ L $. Consider the scheme $Z= \sum_{i=1}^n{m_iP_i}$,  where the multiplicities $m_1\leq m_2\leq \dots \leq m_n$ are nonnegative integers. In this section, we determine the resurgence and we prove Theorem \ref{theo}.
To do so requires some lemmas.

The following lemma plays a significant role throughout this section. 
\begin{lemma} \label{lemma}
Let $F \in R$ be a homogeneous form of degree $d$.
Then there are uniquely determined forms $g_{d,i_2,\ldots,i_N}\in \mathbb{K}[x_0,x_1]$ of degree $d-(i_2+\dots+i_N)$ such that
\begin{equation}\label{equation1}
F=\sum_{k=0}^d \sum_{i_2+\dots+i_N=k}{g_{d,i_2,\ldots,i_N}\cdot x_2^{i_2}\cdots x_N^{i_N}}.
\end{equation}
Moreover, given any homogeneous linear form $G=bx_0+ax_1\ (a,b \in \mathbb{K}\ \mbox{not both zero})$, let $I$ be the ideal $\langle G,x_2,\ldots,x_N \rangle^m.$
Then $F \in I $ if and only if $G^{m-(i_2+\dots+i_N)}$ divides $g_{d,i_2,\ldots,i_N}$ whenever $m>i_2+\dots+i_N.$
\end{lemma}
\begin{proof}
The claim about $F=\sum_{k=0}^d \sum_{i_2+\dots+i_N=k}{g_{d,i_2,\ldots,i_N}\cdot x_2^{i_2}\cdots x_N^{i_N}}$ follows from thinking of $R$ as $R=\mathbb{K}[x_0,x_1][x_2,\ldots,x_N].$ The second claim, regarding $F \in I$, is clear when $a=0$ or $b=0$, taking into account that $I$ is a monomial ideal in these cases. If $G=bx_0+ax_1, a,b\neq0,$ consider the $\mathbb{K}$-algebra automorphism $f:R \to R$ defined by $f(x_i)=x_i$ for all $i\neq1$ with $f(x_1)=G$. Then $f(\langle x_1,\ldots,x_N\rangle^m)=I.$ Taking $\phi$ to be the inverse automorphism, we have 
$$\phi(F)=\sum_{k=0}^d \sum_{i_2+\dots+i_N=k}{\phi(g_{d,i_2,\ldots,i_N})\cdot x_2^{i_2}\cdots x_N^{i_N}} \in \langle x_1,\ldots,x_N\rangle^m,$$
so $x_1^{m-(i_2+\dots+i_N)}$ divides $\phi(g_{d,i_2,\ldots,i_N})$ whenever $m>i_2+\dots+i_N$, hence $G^{m-(i_2+\dots+i_N)}$ divides $g_{d,i_2,\ldots,i_N}$whenever $m>i_2+\dots+i_N$.
\end{proof}
\begin{remark}
Considering the previous proof, since the ideal of the point $P=[-a:b:0:\dots:0]$ is $G=\left\langle bx_0+ax_1,x_2,\ldots,x_N\right\rangle $, indeed, we showed that
$F\in I(mP)$ if and only if  $G^{m-(i_2+\dots+i_N)}|g_{d,i_2,\ldots,i_N}$  whenever $m>i_2+\dots+i_N$.
\end{remark}
Using unique factorization for homogeneous polynomials in $ \mathbb{K}[x_0, x_1]$, the following
corollary is an immediate consequence of the previous lemma.
\begin{corollary} \label{Corollary2.2}
Given distinct points $P_i=[-d_i:c_i:0:\dots:0]$, $i=0,\ldots,n$, on the line $x_2=x_3=\ldots=x_N=0$, let $ F $ be a form as {\em(\ref{equation1})}. Then
$F\in I(\sum_{i=0}^nmm_iP_i)$ if and only if $ {(c_ix_0+d_ix_1)}^{mm_i-(i_2+\dots+i_N)}|g_{d,i_2,\ldots,i_N}$  whenever $i_2+\dots+i_N<mm_i$ for all $i=1,\ldots,n$.
In other words, we have shown that the homogeneous ideal $I(\sum_{i=0}^n{mm_iP_i})$ is generated by ``monomials''  of the type $ G_1^{a_1}\cdots G_n^{a_n}\cdot x_2^{b_2}\cdots x_N^{b_N}$, where $G_j=c_jx_0+d_jx_1$ and  $ a_j= \max_j(0,mm_j-(b_2+\dots+b_N)), \quad j=1,\ldots,n $ and $b_2+\dots+b_N \leq \max(mm_0,\ldots,mm_n)$.
\end{corollary}
The following general lemma gives us a simple criterion for an ideal $I(Z)$ of a fat point scheme to be such that $I(Z)^{(m)}=I(Z)^m$ for all $m \in \mathbb{N}$.
\begin{lemma} \label{criterion}
Let $Z=Z_1+\dots+Z_r$ where $Z_1,\ldots,Z_r \subset \mathbb{P}^N$ are fat point subschemes such that
\begin{equation} \label{crit}
I(kZ)=\prod_{i=1}^rI(kZ_i) \qquad \forall k \in \mathbb{N},
\end{equation}  
where $Z_i$ is a fat point scheme satisfying the condition $I(Z_i)^{(m)}=I(Z_i)^m$, for all $m \in \mathbb{N}$.
Then we have also $$ I(Z)^{m}=I(Z)^{(m)}  \qquad \forall m \in \mathbb{N}.$$
\end{lemma}
\begin{proof}
Considering (\ref{crit}) when $k=1$, we obtain $I(Z)=\prod_{i=1}^r{I(Z_i)}$, thus
$$ I(Z)^m=\prod_{i=1}^r{I(Z_i)^m}=\prod_{i=1}^r{I(Z_i)^{(m)}}=\prod_{i=1}^rI(mZ_i)=I(mZ)=I(Z)^{(m)}.$$ 	
So the proof is complete.
\end{proof}
Taking  into account Lemma \ref{criterion}, in order to prove Theorem \ref{theo}, it suffices to exhibit a suitable splitting for an ideal $I(Z)$ of a collinear fat point scheme. The following lemma gives us a precise answer to this problem.
\begin{lemma} \label{claim}
	Let $Z=\sum_{i=1}^n{m_iP_i}$ be a fat point scheme, where the $P_i$'s are collinear points in $\mathbb{P}^N.$ 	We can assume that the points lie on the line $x_2=x_3=\ldots=x_N=0$  and  $0=m_0\leq m_1\leq \dots \leq m_n$.
	Then 
	$$ I(mZ)=\prod_{i=1}^{n}I((mm_i-mm_{i-1})Z_i), $$
	where $Z_i=P_i+\dots+P_n $ for $i=1,\ldots,n$.
\end{lemma} 
\begin{proof}
		Notice that the ideal $I(Z_i)$ defined in the previous lemma satisfies 
	\[	
	I(Z_i)^{(m)}=I(Z_i)^m \textrm{ for all } m.
	\]
	In fact,   $I(Z_i)$ is a complete intersection scheme (a set of simple points on a line), and by \cite[Lemma 5 and Theorem 2 of Appendix 6]{Zariski1960}, its symbolic powers and ordinary powers are always equal. 
	
	Therefore it is enough to show  $I(mZ)=\prod_{i=1}^{n}I(Z_i)^{mm_i-mm_{i-1}}.$
 We denote by $G_i$ the linear form in $ \mathbb{K}[x_0,x_1] $ such that we have $I(P_i)=\langle G_i,x_2,\ldots,x_N \rangle $ for all $i=1,\ldots,n$. The inclusion $``\supseteq" $ is immediately concluded from the definition of $ I(Z) $. For proving the other inclusion $``\subseteq" $, it suffices to consider Corollary \ref{Corollary2.2} and show that a monomial  $ \mathcal{M}=G_1^{a_1}\cdots G_n^{a_n}\cdot x_2^{b_2}\cdots x_N^{b_N}$ where $ a_j=\max_j(0,mm_j-\sum_{i=2}^Nb_i)$ and $\sum_{i=2}^Nb_i\leq mm_n$, for all $1\leq j\leq n$ is contained in $\prod_{i=1}^{n}I((mm_i-mm_{i-1})Z_i) $.
Regard $H=x_2^{b_2} \cdots x_N^{b_N}$ as a product of $b_2+\cdots+b_N$ linear forms. Let $H_1$ be the product of the first $mm_1$ forms in $H$, $H_2$ be the product of the next $mm_2-mm_1$ linear forms in $H$, etc., until, for some $j$, $H_j$ is the product of the remaining forms in $H$. Since $b_2+\cdots+b_N \leq mm_n$, we know $j \leq n$. If $j<n$, set $H_i=1$ for $i>j$ (in particular, if $b_2+\cdots+b_N<mm_1$, then $H_1=H$ and $H_i=1$ for $1<i\leq n$). 
Define $\mathcal{M}_i=G_i\cdots G_n$ for $i=1,\ldots,n$ and then we can write
$$ \mathcal{M}=(\mathcal{M}_1^{a_1}H_1)(\mathcal{M}_2^{a_2-a_1}H_2)\cdots (\mathcal{M}_n^{a_n-a_{n-1}}H_n), $$
and it is easy to check that $\mathcal{M}_i^{a_i-a_{i-1}}H_i \in I(Z_i)^{(mm_i-mm_{i-1})}$ for each $i$.
\end{proof}
\section{Three noncollinear points: $ \mathbb{P}^N $ versus $ \mathbb{P}^2 $} \label{section3}
\begin{lemma}
Let $ Z $ be a three non-collinear fat points scheme. If $I_{\mathbb{P}^N}(mZ)\subseteq I_{\mathbb{P}^N}(Z)^r $, then $I_{\mathbb{P}^2}(mZ)\subseteq I_{\mathbb{P}^2}(Z)^r $. 
\end{lemma}
\begin{proof}
We have the canonical ring quotient $ q:R_{\mathbb{P}^N}\to R_{\mathbb{P}^2} $. The key fact is that $ q(I_{\mathbb{P}^N}(Z))=I_{\mathbb{P}^2}(Z) $. Hence, if $ I_{\mathbb{P}^N}(mZ)\subseteq I_{\mathbb{P}^N}(Z)^r $, then $ I_{\mathbb{P}^2}(mZ)=q(I_{\mathbb{P}^N}(mZ))\subseteq q(I_{\mathbb{P}^N}(Z)^r)=I_{\mathbb{P}^2}(Z)^r. $
\end{proof}
\begin{corollary}\label{boundd}
$$\rho(I_{\mathbb{P}^2}(Z))\leq \rho(I_{\mathbb{P}^N}(Z))$$
\end{corollary}
\begin{proof}
By the previous lemma it follows that
$$ \left\{m/r: I_{\mathbb{P}^2}(Z)^{(m)}\not \subseteq I_{\mathbb{P}^2}(Z)^r \right\} \subseteq \left\{m/r: I_{\mathbb{P}^N}(Z)^{(m)}\not \subseteq I_{\mathbb{P}^N}(Z)^r \right\},$$
so the desired result easily follows from the definition of resurgence and from the properties of the supremum.
\end{proof}
\begin{proposition} \label{prop33}
Let $Z = m_0P_0 + m_1P_1 + m_2P_2\subset \mathbb{P}^N$, assuming $ \max(m_0,m_1) \leq m_2 $ and that the points are noncollinear. Then $ \alpha(I_{\mathbb{P}^2}(Z)) $ is as follows:
\begin{itemize}
\item[\textbf{(a)}] $ m_2 $ if $ m_2\geq m_1+m_0 $
\item[\textbf{(b)}] $ (m_0 + m_1 + m_2)/2$ if $  m_2\leq m_0 + m_1$ and $m_0 + m_1 + m_2$ is even
\item[\textbf{(c)}] $ (m_0 +m_1 +m_2 +1)/2 $ if $ m_2\leq m_0 +m_1 $ and $ m_0 +m_1 +m_2 $ is odd.
\end{itemize}
\end{proposition}
\begin{proof}
We may choose coordinates so that the points $P_0, P_1, P_2$ are the coordinate vertices of $\mathbb{P}^2$. Namely  we assume that $P_0=[1:0:0]$, $P_1=[0:1:0]$ and $P_2=[0:0:1]$. 
The proof in case \textbf{(a)} is: $ x_0^{m_2-m_0}x_1^{m_0} \in I_{\mathbb{P}^2}(Z) $ hence $ \alpha(I_{\mathbb{P}^2}(Z))\leq m_2 $, but no non-zero form of degree less than $ m_2 $ can vanish to order $ m_2 $ at a point, hence
$ \alpha(I_{\mathbb{P}^2}(Z))\geq m_2  $ too. 
The proof in case \textbf{(b)} is: 
$$ x_0^{(m_2+m_1-m_0)/2}x_1^{(m_2+m_0-m_1)/2}x_2^{(m_1+m_0-m_2)/2}\in I_{\mathbb{P}^2}(Z) $$
 so $ \alpha(I_{\mathbb{P}^2}(Z))\leq(m_0+m_1+m_2)/2 $. But $I_{\mathbb{P}^2}(Z)$ is monomial and
 there are irreducible conics through the three points. Thus $ 2\alpha(I_{\mathbb{P}^2}(Z)) \geq m_0+m_1+m_2$ by Bezout's Theorem. Thus, $ \alpha(I_{\mathbb{P}^2}(Z))=(m_0+m_1+m_2)/2 $. The proof in the last case is:  all three of $ m_2+m_1-m_0$, $m_2 +m_0-m_1$ and $m_1 +m_0-m_2 $ are odd and nonnegative, hence at least one.  Then 
$$ x_0^{(m_2+m_1-m_0+1)/2}x_1^{(m_2+m_0-m_1+1)/2}x_2^{(m_1+ m_0-m_2-1)/2}\in I_{\mathbb{P}^2}(Z), $$
 so $ \alpha (I_{\mathbb{P}^2}(Z))\leq(m_0+m_1+m_2+1)/2 $. But as before there are irreducible conics through the three points. Thus $ 2\alpha (I_{\mathbb{P}^2}(Z))\geq m_0+m_1+m_2$ by B\'{e}zout's Theorem (as before), and thus $ 2\alpha (I_{\mathbb{P}^2}(Z))\geq m_0+m_1+m_2+1$ (since $m_0+m_1+m_2$ is odd). Thus $ \alpha (I_{\mathbb{P}^2}(Z))= (m_0+m_1+m_2+1)/2 $.
\end{proof}
\begin{proposition} \label{prop}
	Let $P_0,P_1$ and $P_2$ be three noncollinear  points in $\mathbb{P}^N$ and consider  $Z= m_0P_0+m_1P_1+ \nolinebreak m_2P_2$. Suppose $m_0 \leq m_1 \leq m_2$. If $m_0+m_1 > m_2$ and $\sum_{i=0}^2m_i$  is odd, then  $\rho(I_{\mathbb{P}^N}(Z)) \geq \frac{1+\sum_{i=0}^2{ m_i}}{\sum_{i=0}^2 m_i}. $
\end{proposition}
\begin{proof}
	The points $P_0,P_1, P_2$ span a plane  $\mathbb{P}^2 \subset \mathbb{P}^N $. Without loss of generality we assume in this $\mathbb{P}^2$ that $P_0=[1:0:0]$, $P_1=[0:1:0]$ and $P_2=[0:0:1]$. We want to use the following inequality 
	\begin{equation} \label{bound1} \rho(I_{\mathbb{P}^2}(Z)) \geq \alpha(I_{\mathbb{P}^2}(Z))/ \widehat{\alpha}(I_{\mathbb{P}^2}(Z)), 
	\end{equation}
	which was proved in \cite[Theorem 1.2]{bocci2010resurgence}.
	From the part \textbf{(c)} of the last proposition we have $ \alpha (I_{\mathbb{P}^2}(Z))= (m_0+m_1+m_2+1)/2 $.
	Now, consider $m \in \mathbb{N}$ and the $2m$-th symbolic power $I_{\mathbb{P}^2}(Z)^{(2m)}$. Considering the definition of symbolic powers, 
	$$ I_{\mathbb{P}^2}(Z)^{(2m)}= \langle x_1,x_2\rangle^{2m \cdot m_0} \cap \langle x_0,x_2\rangle^{2m \cdot m_1} \cap \langle x_0,x_1\rangle^{2m \cdot m_2}.$$
Since $I_{\mathbb{P}^2}(Z)^{(2m)}=I_{\mathbb{P}^2}(2mZ)$ we have $\alpha(I_{\mathbb{P}^2}(Z)^{(2m)})=m(m_0+m_1+m_2)$ by Proposition \ref{prop33} (b). 	 Thus we obtain the Waldschmidt constant of $I_{\mathbb{P}^2}(Z)$ as follows:
	\begin{flalign} \label{const}
	 \widehat{\alpha}(I_{\mathbb{P}^2}(Z))&= \lim_{m \to \infty }{\frac{\alpha(I_{\mathbb{P}^2}(Z)^{(m)})}{m}}=\lim_{m \to \infty }{\frac{\alpha(I_{\mathbb{P}^2}(Z)^{(2m)})}{2m}}\\\nonumber
	  &= \lim_{m \to \infty }{\frac{m(\sum_{i=0}^2m_i)}{2m}}=\frac{\sum_{i=0}^2m_i}{2}.
	\end{flalign}  
	Hence by (\ref{const}) and (\ref{bound1}), 
	$$ \frac{1+\sum_{i=0}^2m_i}{\sum_{i=0}^2m_i} \leq  \frac{\alpha(I_{\mathbb{P}^2}(Z))}{\widehat{\alpha}(I_{\mathbb{P}^2}(Z))} \leq \rho(I_{\mathbb{P}^2}(Z)),$$
	and by Corollary \ref{boundd} the desired result is obtained.
\end{proof}
\section{Three noncollinear points in $\mathbb{P}^N$} \label{section4}
In this section we obtain additional results for the fat point scheme $Z=m_0P_0+m_1P_1+m_2P_2$ in $\mathbb{P}^N$, 
where $P_0, P_1$ and $P_2$ are noncollinear and each $m_i$ is a nonnegative integer. We can assume $P_0=[1:0:0:\dots : 0]$, $P_1=[0:1:0: \dots:0]$ and $P_2=[0:0:1:0: \dots:0]$ and $ \max(m_0,m_1) \leq m_2 $. Notice that the ideal $ I(P_i) $ is a square-free monomial ideal, and hence $ I(Z) $ is a monomial ideal. We are interested in computing the resurgence $\rho(I(Z))$ of the ideal $I(Z)$. In particular, we want to understand how the resurgence of the scheme $ Z $ depends on the values of the multiplicities $m_i$.

The following lemma gives some conditions for a monomial to belong to $I(Z)$.
\begin{lemma}
\label{general}
Let $P_0, P_1$ and $P_2$ be noncollinear points in $\mathbb{P}^N$ as above and $m_i \geq 0$. We define the fat point scheme $Z=m_0P_0+m_1P_1+m_2 P_2$. Then the monomial $\mathcal{N}=x_0^{a_0}x_1^{a_1}\cdots x_N^{a_N} \in I(Z)$ if and only if $(a_0,\ldots,a_N)$ satisfies the following system of inequalities
\begin{equation} 
\label{condiz}  
\operatorname{Cond}(Z):=\begin{cases}
a_1+a_2+a_3+\cdots+a_N \geq m_0 \\ a_0+a_2+a_3+\cdots+a_N \geq m_1 \\ a_0+a_1+a_3+\cdots+a_N \geq m_2.
	\end{cases}\end{equation}
\end{lemma}
\begin{proof}
The result easily follows from the fact that the ideal $I(Z)$ is the monomial ideal
$\bigcap_{i=0}^2 I(P_i)^{m_i}$ with $ I(P_0) = (x_1,x_2,x_3,...,x_N)$, $I(P_1) = (x_0,x_2,x_3,...,x_N)$ and $I(P_2) = (x_0,x_1,x_3,...,x_N) $.
\end{proof}
Notice that in the previous lemma, in order to simplify the notation, we made implicit the dependence of $ \operatorname{Cond}(Z) $ on $ m_0$, $m_1$ and $m_2 $.

We divide this section into two subsections where we study distinct configurations for the multiplicities $m_i$.
\subsection{Case $m_0+m_1\leq m_2$}
The aim of this subsection is to prove the following result.
\begin{proposition} \label{cor1}
	Let $P_0, P_1$ and $P_2$ be noncollinear points in $\mathbb{P}^N$ and $ m_2 \geq \max(m_0,m_1)$. Let $Z=\sum_{i=0}^2{m_iP_i} $ be a fat point scheme. If $m_0+m_1\leq m_2$, then $I(Z)^{(m)}=I(Z)^m \textrm{ for all } m \in \mathbb{N}$ and consequently $\rho(I(Z))=1$.	
\end{proposition}
  Proposition \ref{cor1} follows at once by accordingly using  Lemma \ref{criterion} if we can find a suitable splitting for the ideal $I(Z)$. As explained in Remark \ref{remark3} (following the proof of Lemma \ref{theo3} ), the following lemma gives a suitable splitting.
\begin{lemma} \label{theo3}
	Let $P_0, P_1$ and $P_2$ be noncollinear points in $\mathbb{P}^N$ and $m_2\geq\max(m_0,m_1). $ 
Consider the fat point scheme $Z=m_0P_0+m_1 P_1+m_2P_2 $.  If $m_0+m_1\leq m_2$, then $$ I(Z)=I(m_0(P_0+P_2)) \cdot I(m_1(P_1+P_2))\cdot I((m_2-m_0-m_1)P_2).$$
\end{lemma}
\begin{proof}
	Notice that, if $ m_2=m_0+m_1 $, then $ I((m_2-m_0-m_1)P_2)=R $ thus the desired splitting  in this case is $ I(Z)=I(m_0(P_0+P_2)) \cdot I(m_1(P_1+P_2))$.
Set $Z_1=m_0(P_0+P_2)$, $Z_2=m_1(P_1+P_2)$ and $Z_3=(m_2-m_0-m_1)P_2$. The inclusion $ I(Z_1) \cdot I(Z_2)\cdot I(Z_3) \subseteq I(Z)$ is trivial since $Z=Z_1+Z_2+Z_3$. Now, we show the other inclusion holds.
Thus, let us consider a  monomial $\mathcal{N}=x_0^{a_0}x_1^{a_1}\cdots x_N^{a_N} \in I(Z)$ where the $a_i$'s satisfy the system $\textrm{Cond(Z)}$, and set $b=\sum_{i=3}^N{a_i}.$ We have the following cases:\\
\textbf{(a)} Assume $a_1+b < m_0$. By $\textrm{Cond(Z)}$ it follows that $a_2 \geq m_0-a_1-b >0 $ and $a_0 \geq m_2-a_1-b= (m_0-a_1-b)+m_1+(m_2-m_0-m_1)$. Then the monomial
$$  (x_0^{m_0-a_1-b}x_1^{a_1}x_2^{m_0-a_1-b}x_3^{a_3}\cdots x_N^{a_N}) \cdot (x_0^{m_1})\cdot (x_0^{m_2-m_0-m_1}) $$ divides $\mathcal{N}$ and belongs to $ I(Z_1)\cdot I(Z_2)\cdot I(Z_3)$ because the $j$-th part of the above product is in $I(Z_j)$ by $\textrm{Cond}(Z_j)$ for $j=1,2,3$.\\
\textbf{(b)} Assume that $ a_0+b<m_1$. The proof is similar to the previous case by using $ a_0+b<m_1$.\\
\textbf{(c)} Consider $a_1+b \geq m_0$ and $ a_0+b \geq m_1$ and the following four cases:
\begin{itemize}
\item[\textbf{(1)}] Assume $ a_1 \geq m_0$ and $ a_0\geq m_1$. We can write $ \mathcal{N} $ as
$$(x_1^{m_0}) \cdot (x_0^{m_1})\cdot (x_0^{a_0-m_1}x_1^{a_1-m_0}x_2^{a_2}x_3^{a_3}\cdots x_N^{a_N}).$$
The first two factors belong respectively to $I(Z_1)$ and $I(Z_2)$ while the third one is in $ I(Z_3) $ because $\operatorname{Cond}(Z)$ implies $  a_0+a_1+a_3+\cdots+a_N-m_1-m_0 \geq m_2-m_1-m_0 $.
Hence, $\textrm{Cond}(Z_3)$ is satisfied.\\
\item[\textbf{(2)}] Assume $a_1<m_0$ and $a_0 \geq m_1$. By $ a_1+b \geq m_0$, we deduce 
$ \sum_{i=3}^N{a_i} \geq m_0-a_1. $
Then, for each $i=3,\ldots,N$, we can choose $0\leq b_i \leq a_i$ such that $ \sum_{i=3}^{N}{b_i}=\nolinebreak m_0-a_1. $
It can be written
$ \mathcal{N}=(x_1^{a_1}x_3^{b_3}\cdots x_{N}^{b_N}) \cdot (x_0^{m_1})\cdot (x_0^{a_0-m_1}x_2^{a_2}x_3^{a_3-b_3}\cdots x_N^{a_N-b_N}),  $ where it is easy to check that the first two factors belong respectively to $I(Z_1)$ and $I(Z_2)$ while
the third term is in $I(Z_3)$ because $\textrm{Cond}(Z)$ implies that
$a_0+a_3+\cdots+a_N-m_1-\sum_{i=3}^{N}{b_i}=a_0+a_1-m_1-m_0+\sum_{i=3}^N{a_i} \geq m_2-m_1-m_0 $. So, $\textrm{Cond}(Z_3)$ is satisfied.
\item[\textbf{(3)}]Assume $a_1 \geq m_0$ and $a_0 < m_1$. The proof of this case is similar to the proof of the previous one.
\item[\textbf{(4)}]
Assume $a_1 < m_0$ and $a_0 < m_1$. $\textrm{Cond}(Z)$ implies that
\begin{flalign*}
&\sum_{i=3}^{N}{a_i}=b\\
&=(a_0+a_1+b-m_0-m_1)+(m_0-a_1)+(m_1-a_0)\\
&\geq (m_2-m_1-m_0)+(m_0-a_1)+(m_1-a_0).
\end{flalign*}
Because the  last three summands are all positive we can choose for all $i=3,\ldots,N$, some integers $ 0\leq c_i,d_i,e_i \leq a_i$ such that 
$c_i+d_i+e_i \leq a_i$ for all $ i=3,\ldots,N $, $ \sum_{i=3}^{N}{d_i}=m_1-a_0$, $\sum_{i=3}^{N}{c_i}=m_0-a_1$, and $ \sum_{i=3}^{N}{e_i}=m_2-m_1-m_0.$ So, the monomial $ \mathcal{M}= (x_1^{a_1}x_3^{c_3}\cdots x_{N}^{c_N}) \cdot (x_0^{a_0}x_3^{d_3}\cdots x_{N}^{d_N})\cdot (x_3^{e_i}\cdots x_N^{e_N})   $ divides $\mathcal{N}$ and belongs to $ I(Z_1)\cdot I(Z_2)\cdot I(Z_3)$ because its $j$-th factor belongs to $I(Z_j)$ by $\textrm{Cond}(Z_j)$ for $j=1,2,3$.
\end{itemize}
\end{proof}
\begin{remark} \label{remark3}
Notice that the splitting presented in the previous lemma satisfies the condition of Lemma \ref{criterion}. In fact the ideals involved in the product are ideals of   fat point schemes whose support consists of collinear points, and by means of Theorem \ref{theo} we  have 
$$	I(Z_i)^{(m)}=I(Z_i)^m \textrm{ for all } m.$$
Furthermore,  Lemma \ref{theo3}  can be applied to the fat point scheme $kZ$ where $k \in \mathbb{N}$, deducing that
$$ I(kZ)=\prod_{i=1}^3I(kZ_i).$$
\end{remark}
\subsection{Case $m_0+m_1 > m_2$}
In this subsection, we deal with the case $m_0+m_1 > m_2$ showing how the value of the resurgence depends on the parity of the sum $\sum_{i=0}^2{m_i}$. 
Using the same approach as in  the previous subsection, we want to split the ideal $I(Z)$ in a convenient way as a product of ideals $I(Z_i)$.
\begin{lemma} \label{theo5}
Let $P_0, P_1$ and $P_2$ be noncollinear points in $\mathbb{P}^N$ and $m_2 \geq \max(m_0,m_1)$. We consider the scheme $Z=m_0P_0+m_1 P_1+m_2P_2$. If $m_0+m_1 > m_2$, then $ I(Z)= I(Z_1) \cdot I(Z_2)\cdot I(Z_3) $ where,
\begin{itemize}
	\item $Z_1=(m_0+m_1-m_2)(P_0+P_1+P_2)$
	\item $Z_2=(m_2-m_1)(P_0+P_2)$
	\item $ Z_3=(m_2-m_0)(P_1+P_2)$.
\end{itemize}
\end{lemma}
\begin{proof}
The inclusion $ I(Z_1) \cdot I(Z_2)\cdot I(Z_3) \subseteq I(Z)$ is trivial since $Z=Z_1+Z_2+Z_3$. We just need to show that if a monomial $\mathcal{N}=x_0^{a_0}x_1^{a_1}\cdots x_N^{a_N} \in I(Z), $ then $\mathcal{N}\in I(Z_1)\cdot I(Z_2)\cdot I(Z_3)$.
Thus, suppose $\mathcal{N} \in I(Z)$, and set $b= \sum_{i=3}^N{a_i}.$ We have the following cases:\\
\textbf{(a)} Let $a_1+b<m_2-m_1$. Considering the system $\textrm{Cond}(Z)$,
\begin{itemize}
\item $ a_2 \geq m_0-a_1-b=(m_0+m_1-m_2)+(m_2-m_1-a_1-b) $ 
\item $ a_0 \geq m_2-a_1-b=(m_0+m_1-m_2)+(m_2-m_1-a_1-b)+(m_2-m_0)$
\end{itemize}
where all the numbers between parenthesis are nonnegative. So,  the  monomial $\mathcal{M}=\left((x_0x_2)^{m_0+m_1-m_2}\right)  \cdot \left((x_0x_2)^{m_2-m_1-a_1-b}x_1^{a_1}x_3^{a_3}\cdots x_{N}^{a_N}\right)\cdot \left(x_0^{m_2-m_0} \right)   $
divides $\mathcal{N}$. Furthermore, $\mathcal{M}$ belongs to $ I(Z_1)\cdot I(Z_2)\cdot I(Z_3)$ because the $j$-th factor belongs to $I(Z_j)$ by $\textrm{Cond}(Z_j)$ for $j=1,2,3$. Thus $\mathcal{N} \in I(Z_1)\cdot I(Z_2)\cdot I(Z_3)$.\\
\textbf{(b)}   $a_0+b < m_2-m_0$: the proof is similar to previous case using $ a_0+b<m_2-m_0$.\\
\textbf{(c)}  $ a_1+b \geq m_2-m_1$ and $a_0+b \geq m_2-m_0$: we have four subcases,
\begin{itemize}
\item[\textbf{(1)}]$a_1 \geq m_2-m_1$ and $a_0 \geq m_2-m_0$: we can write
 $\mathcal{N}=(x_0^{a_0-m_2+m_0}x_1^{a_1-m_2+m_1}x_2^{a_2}x_3^{a_3}\cdots x_N^{a_N}) \cdot (x_1^{m_2-m_1})\cdot ( x_0^{m_2-m_0})$ where the first factor is in $I(Z_1)$ because $\textrm{Cond}(Z)$ implies
 \begin{itemize}
 	\item[$\bullet$] $ a_1+a_2+a_3+\cdots+a_N-m_2+m_1 \geq m_0+m_1-m_2 $
 	\item[$\bullet$] $ a_0+a_2+a_3+\cdots+a_N-m_2+m_0 \geq m_0+m_1-m_2 $
 	\item[$\bullet$] $ a_0+a_1+a_3+\cdots+a_N-2m_2+m_0+m_1 \geq m_0+m_1-m_2 $
 \end{itemize}
So, $\textrm{Cond}(Z_1)$ is satisfied. Furthermore, it is easy to check that $x_1^{m_2-m_1} \in I(Z_2)$ and $x_0^{m_2-m_0} \in \nolinebreak I(Z_3)$. Thus $\mathcal{N} \in  I(Z_1)\cdot I(Z_2)\cdot I(Z_3)$.
\item[\textbf{(2)}] $a_1 \geq m_2-m_1$ and $a_0 < m_2-m_0$: $ a_0+b \geq m_2-m_0$ implies that $ \sum_{i=3}^N{a_i} \geq m_2-m_0-a_0>0 $.
For each $i=3,\ldots,N$ we can choose $0\leq b_i \leq a_i$ such that $ \sum_{i=3}^{N}{b_i}=m_2-m_0-a_0. $
We can write
$ \mathcal{N}=(x_1^{a_1-m_2+m_1}x_2^{a_2}x_3^{a_3-b_3}\cdots x_{N}^{a_N-b_N}) \cdot (x_1^{m_2-m_1})\cdot (x_0^{a_0}x_3^{b_3}\cdots x_N^{b_N}),  $
where the first factor is in 
$ I(Z_1)$ because by $\textrm{Cond}(Z)$, it follows
\begin{itemize}
	\item[$\bullet$] $a_1+a_2+b-m_2+m_1-\sum_{i=3}^N{b_i}=-2m_2+m_0+m_1+\sum_{i=0}^{N}{a_i} \geq  m_0+m_1-m_2  $
	\item[$\bullet$] $ a_2+b-\sum_{i=3}^N{b_i} =a_0+m_0-m_2+\sum_{i=2}^N{a_i} \geq m_0+m_1-m_2 $
	\item[$\bullet$] $ a_1+b+m_1-m_2-\sum_{i=3}^N{b_i}=a_0+a_1 +m_0+m_1-2m_2+ \sum_{i=3}^{N}{a_i} \geq m_0+m_1-m_2. $
\end{itemize}
So, the conditions at $\textrm{Cond}(Z_1)$ are satisfied. As we have seen in the previous subcase, the second factor belongs to $I(Z_2)$. Furthermore, it is easy to check, using $\textrm{Cond}(Z_3)$, that $x_0^{a_0}x_3^{b_3}\cdots x_N^{b_N} \in \nolinebreak I(Z_3)$. Thus $\mathcal{N} \in  I(Z_1)\cdot I(Z_2)\cdot I(Z_3)$.
\item[\textbf{(3)}] $a_1 < m_2-m_1$ and $a_0 \geq m_2-m_0$: the proof  is similar to the previous one.
\item[\textbf{(4)}] $a_1 < m_2-m_1$ and $a_0 < m_2-m_0$: by $\textrm{Cond}(Z)$, we have	
$b=\sum_{i=3}^{N}{a_i}
=(a_1+a_0+b+m_0+m_1-2m_2)+(m_2-m_1-a_1)
+(m_2-m_0-a_0)\geq (m_0+m_1-m_2)+(m_2-m_1-a_1)
+(m_2-m_0-a_0)$. Because the last three summands are all positive, it is possible to choose for all $i=3,\ldots,N$ some integers $ 0\leq c_i, d_i, e_i \leq a_i$ such that $ c_i+d_i+e_i \leq a_i$ for all $ i=3,\ldots,N$, $ \sum_{i=3}^{N}{c_i}=m_0+m_1-m_2 $, $ \sum_{i=3}^{N}{d_i}=m_2-m_1-a_1$ and $ \sum_{i=3}^{N}{e_i}=m_2-m_0-a_0 $.
By $\textrm{Cond}(Z_i)$, it  follows that 
$$ \mathcal{M}=(x_3^{c_3}\cdots x_{N}^{c_N}) \cdot (x_1^{a_1}x_3^{d_3}\cdots x_{N}^{d_N})\cdot (x_0^{a_0}x_3^{e_i}\cdots x_N^{e_N}) $$ belongs to $ I(Z_1)\cdot I(Z_2)\cdot I(Z_3). $ Since $\mathcal{M}$ divides $\mathcal{N}$, we deduce $\mathcal{N} \in I(Z_1)\cdot I(Z_2)\cdot \nolinebreak I(Z_3)$. 
\end{itemize}
So, in all the possible cases,  $\mathcal{N} \in  I(Z_1)\cdot I(Z_2)\cdot I(Z_3)$ and
the proof of the lemma is complete.\qedhere
\end{proof}
The next lemma helps us to deal with subschemes of the type $(2q+r)(P_0+P_1+P_2)$  that appeared as a factor in the splitting presented in Lemma \ref{theo5}.
\begin{lemma} \label{lemma31}
	Let $P_0, P_1$ and $P_2$ be three noncollinear points in $\mathbb{P}^N$. If  $q,r \in \mathbb{N}$ with $0\leq r <2$, then
	$ I((2q+r)(P_0+P_1+P_2))=I(2(P_0+P_1+P_2))^q\cdot I(P_0+P_1+P_2)^r.$
\end{lemma}
\begin{proof}
We give a proof by induction on $q$. In order to prove the base case $q=0$, we need to show that
$ I(r(P_0+P_1+P_2))=I(P_0+P_1+P_2)^r$.
But this is trivial for $r=0,1$. For the induction, we suppose that the lemma is true for $q-1$ and we prove that it holds for $q$. We claim that 
$$ I((2q+r)(P_0+P_1+P_2))=I(2(P_0+P_1+P_2))\cdot I((2(q-1)+r)(P_0+P_1+P_2)).$$
\noindent \textbf{Proof of the claim.} \\Set $ 
Z_1=(2q+r)(P_0+P_1+P_2)$, $Z_2=2(P_0+P_1+P_2)$ and $ Z_3=(2(q-1)+r)(P_0+P_1+P_2)$. The inclusion $ I(Z_2) \cdot I(Z_3) \subseteq I(Z_1)$ is trivial from the definition. Therefore, we  show that if a monomial $\mathcal{N}=x_0^{a_0}x_1^{a_1}\cdots x_N^{a_N} \in I(Z_1) $
then $ \mathcal{N} \in I(Z_2)\cdot I(Z_3)$.
Thus, consider $\mathcal{N} \in I(Z_1)$. Set $b= \sum_{i=3}^N{a_i}.$ We have the following cases:\\
\textbf{(a)} Let $  b \geq 2 $, for each $i=3,\ldots,N$, it can be chosen $0\leq b_i \leq a_i$ such that $ \sum_{i=3}^{N}{b_i}=2. $
If we write
$ \mathcal{N}=(x_3^{b_3}\cdots x_{N}^{b_N}) \cdot (x_0^{a_0}x_1^{a_1}x_2^{a_2}x_3^{a_3-b_3}\cdots x_{N}^{a_N-b_N})$, then we can easily deduce by $\textrm{Cond}(Z_2)$ and $\textrm{Cond}(Z_3)$ that $\mathcal{N} \in  I(Z_2)\cdot I(Z_3)$.\\
\textbf{(b)} Let $b=1$. We have three subcases.
\begin{itemize}
	\item[\textbf{(1)}]
$a_0=0$: by $\textrm{Cond}(Z_1)$ it follows that 
\begin{eqnarray}\label{eq1}
a_1 \geq 2q+r-1 \geq 1  \textrm{ and } a_2 \geq 2q+r-1 \geq 1.
\end{eqnarray}
Therefore,
$ \mathcal{N}=(x_1x_2x_3^{a_3}\cdots x_{N}^{a_N}) \cdot (x_1^{a_1-1}x_2^{a_2-1}) $,
where  it is easy to see that the  first factor is in $I(Z_2)$.  So we need to show that the second one is in $I(Z_3)$. By $\textrm{Cond}(Z_1)$ it follows
\begin{itemize}
\item[$\bullet$]  $ a_j-1 \geq2(q-1)+r $ for $j=1,2$\ by (\ref{eq1})
\item [$\bullet$] $ a_1+a_2-2\geq a_2-1 \geq 2(q-1)+r $\ by (\ref{eq1}).
\end{itemize}
Thus the conditions  at $\textrm{Cond}(Z_3)$ are satisfied.
\item[\textbf{(2)}] $a_1=0$ and $a_2=0$:	these cases are similar to the previous one.
\item[\textbf{(3)}] $ a_0,a_1,a_2>0$:
we can write 
$$ \mathcal{N}=(x_0x_1x_2) \cdot (x_0^{a_0-1}x_1^{a_1-1}x_2^{a_2-1}x_3^{a_3}\cdots x_{N}^{a_N}),  $$ where it is easy to prove that the two factors belong  to $I(Z_2)$ and $I(Z_3)$ respectively.
\end{itemize}
\textbf{(c)} If we consider $ b=0$, then it is a known case in $ \mathbb{P}^2 $ (see the end of section 6 in \cite{boccicooper}). Using the canonical inclusion $ R_{\mathbb{P}^2}\subseteq R_{\mathbb{P}^N} $ we get $ I_{\mathbb{P}^2}(Z_i)\subset I_{\mathbb{P}^N}(Z_i) $. Hence, we have $$ \mathcal{N}\in I_{\mathbb{P}^2}(Z_1)=I_{\mathbb{P}^2}(Z_2)\cdot I_{\mathbb{P}^2}(Z_3)\subset I_{\mathbb{P}^N}(Z_2)\cdot I_{\mathbb{P}^N}(Z_3).$$
So, the proof of the claim is complete. By the inductive step 
\begin{flalign*}
&I((2q+r)(P_0+P_1+P_2))\\
&=I(2(P_0+P_1+P_2))\cdot I((2(q-1)+r)(P_0+P_1+P_2))\\
&=I(2(P_0+P_1+P_2))\cdot I(2(P_0+P_1+P_2))^{q-1}\cdot I(P_0+P_1+P_2)^r,
\end{flalign*}
and the proof  is complete.
\end{proof}
Now, we can solve our main problem when  $\sum_{i=0}^2{m_i}$ is even. 
\begin{proposition} \label{cor7}
Let $P_0, P_1$ and $P_2$ be noncollinear points in $\mathbb{P}^N$ and $m_0 \leq m_1 \leq m_2 $. Denote by $Z$ the corresponding fat point scheme $Z=m_0P_0+m_1 P_1+m_2P_2$. If $m_0+m_1> m_2$ and $m_0+m_1+m_2$ is even, then $ I(Z)^{(m)}=I(Z)^m \textrm{ for all } m \in \mathbb{N}$ and hence $\rho(I(Z))=1$.
\end{proposition}
\begin{proof} Since $m_0+m_1+m_2$ is even, we can set $m_0+m_1-m_2=2q$.
By applying Lemma \ref{theo5} to the fat point scheme $kZ=km_0P_0+km_1P_1+\nolinebreak km_2P_2$, we obtain 
\begin{flalign*}
I(kZ)=  I(k2q(P_0+P_1+P_2))& \cdot I(k(m_2-m_1)(P_0+P_2))\\
&\cdot I(k(m_2-m_0)(P_1+P_2)).
\end{flalign*}
From Lemma \ref{lemma31} we can also deduce
$$ I(2q(P_0+P_1+P_2))^{(m)}=I(2q(P_0+P_1+P_2))^{m} \qquad \forall m\in \mathbb{N}.$$
Thus the conditions of Lemma \ref{criterion} are satisfied for $Z$, and we can deduce the desired result.
\end{proof}
Let $\sum_{i=0}^2{m_i}$ be odd.
Our aim is proving  Theorem \ref{casec}.
Proposition \ref{prop} gives us a suitable lower bound, so we need to prove that $\rho(I(Z))\leq  \frac{1+\sum{m_i}}{\sum{m_i}}$. We  will do it by directly considering the definition of resurgence and using further preliminary lemmas on the splitting of the symbolic powers.
By Lemma  \ref{theo5} and Lemma \ref{lemma31}, we can deduce the following corollary.
\begin{corollary} \label{cor8}
Let $P_0, P_1$ and $P_2$ be noncollinear points in $\mathbb{P}^N$ and $m_0 \leq m_1 \leq m_2 $. Denote by $Z$ the corresponding fat point scheme $Z=m_0P_0+m_1 P_1+m_2P_2 $. 
If $m_0+m_1>m_2$ and $m_0+m_1+m_2$ is odd, then for all $ k \in \mathbb{N}$,\\
$ I(Z)^{(k)}=I(P_0+P_1+P_2)^{(k)}\cdot I((m_0-1)P_0+(m_1-1)P_1+(m_2-1)P_2)^{(k)}= I(P_0+P_1+P_2)^{(k)}\cdot I((m_0-1)P_0+(m_1-1)P_1+(m_2-1)P_2)^{k}$.	
\end{corollary}
\begin{proof}
Consider $k \in \mathbb{N}$. We write $k=2q_1+r$ where $0\leq r<2$. Since $m_0+m_1+m_2$ is odd, it follows
$ m_0+m_1-m_2$ is odd and we can write $ m_0+m_1-m_2=2q_2+1$.
Set $Z_1=P_0+P_1+P_2$.
Because $km_0+km_1=k(m_0+m_1) > km_2$, we can apply Lemma \ref{theo5} to the scheme $kZ$ and we obtain that $I(Z)^{(k)}$ is equal to\\
$I(k(m_0+m_1-m_2)Z_1) \cdot I(k(m_2-m_1)(P_0+P_2))\cdot I(k(m_2-m_0)(P_1+P_2))=I((2(2q_1q_2+q_1+rq_2)+r)Z_1)\cdot I(k(m_2-m_1)(P_0+P_2))\cdot I(k(m_2-m_0)(P_1+P_2))
=I(2Z_1)^{2q_1q_2+q_1+rq_2}\cdot I(Z_1)^r \cdot I(k(m_2-m_1)(P_0+P_2))\cdot I(k(m_2-m_0)(P_1+P_2)),
$
where the last equality holds by Lemma \ref{lemma31}.
By Lemma \ref{lemma31}
$$I(Z_1)^{(k)}=I((2q_1+r)Z_1)= I(2Z_1)^{q_1}\cdot I(Z_1)^r.$$
 By applying Lemma \ref{theo5}  to the scheme 
$Z'=k(m_0-1)P_0+k(m_1-1)P_1+k(m_2-1)P_2$ (we can use it because $k(m_0-1)+k(m_1-1)=k(m_0+m_1-2) \geq k(m_2-1)$ since $m_0+m_1 \geq m_2+1$) we have
$I(Z')=I((m_0-1)P_0+(m_1-1)P_1+(m_2-1)P_2)^{(k)}
= I(k(m_0+m_1-m_2-1)Z_1)\cdot I(k(m_2-m_1)(P_0+P_2))\cdot I(k(m_2-m_0)(P_1+P_2))
=I((2(2q_1q_2+rq_2)Z_1)\cdot I(k(m_2-m_1)(P_0+P_2))\cdot I(k(m_2-m_0)(P_1+P_2))
=I(2Z_1)^{2q_1q_2+rq_2}\cdot I(k(m_2-m_1)(P_0+P_2))\cdot I(k(m_2-m_0)(P_1+P_2)),
$
where the last equality holds by Lemma \ref{lemma31}.
Thus, \begin{align*} 
&I(Z_1)^{(k)}\cdot I((m_0-1)P_0+(m_1-1)P_1+(m_2-1)P_2)^{(k)}\\
&=I(2Z_1)^{q_1}\cdot I(Z_1)^r\cdot I(Z') \\
&=I(2Z_1)^{2q_1q_2+rq_2+q_1}\cdot I(Z_1)^r\cdot I(k(m_2-m_1)(P_0+P_2))\\
& \cdot I(k(m_2-m_0)(P_1+P_2))=I(Z)^{(k)}.
\end{align*}
Finally notice that by Propositions \ref{cor1} and \ref{cor7} it follows that 
\begin{flalign*}
&I((m_0-1)P_0+(m_1-1)P_1+(m_2-1)P_2)^{(k)}\\
&=I((m_0-1)P_0+(m_1-1)P_1+(m_2-1)P_2)^{k}.\qedhere
\end{flalign*}
\end{proof}
Notice that in general the equality  $I(Z)^{(a+b)}=I(Z)^{(a)} \cdot I(Z)^{(b)}$ is not satisfied. 
However, the previous results  imply the following corollary which tells us when this splitting is possible for  $I(Z)$. 
\begin{corollary} \label{cor9}
Let $P_0, P_1$ and $P_2$ be noncollinear points in $\mathbb{P}^N$ and $m_2 \geq \max(m_0,m_1)$. Denote by $Z$ the fat point scheme $Z=m_0P_0+m_1 P_1+m_2P_2. $ If $m_0+m_1>m_2$ and $m_0+m_1+m_2$ is odd, then
\begin{flalign*}
 &I(Z)^{(k)}=I(Z)^{(2i)}\cdot I(Z)^{(k-2i)}\  \mbox{for}\ 1\leq i \leq \frac{k}{2}-1   \ \mbox{if} \  k \  \mbox{is even} \\
 &I(Z)^{(k)}=I(Z)^{(i)}\cdot I(Z)^{(k-i)} \  \mbox{for}\  1\leq i \leq k-1   \ \mbox{if} \  k \  \mbox{is odd},
\end{flalign*}
i.e., $I(Z)^{(k)}=I(Z)^{(i)}\cdot I(Z)^{(k-i)}$ as long as $i$ and $k-i$ are not both odd.
\end{corollary}
\begin{proof}
\textbf{(a)} Suppose that $k=2q$. Then $ I(Z)^{(2q)}=I(2Z)^{(q)}$, where $2Z$ is a fat point scheme that satisfies the condition of the Proposition \ref{cor7}. Therefore \\
$ I(2Z)^{(q)}=I(2Z)^q=I(2Z)^i\cdot I(2Z)^{q-i}=I(2Z)^{(i)}\cdot I(2Z)^{(q-i)}=I(Z)^{(2i)}\cdot I(Z)^{(k-2i)}. $\\
\textbf{(b)} 
Suppose that $k=2q+1$.
By Proposition \ref{cor7}, Lemma  \ref{lemma31}, Corollary \ref{cor8} and the even case, it follows that
\begin{flalign*}
&I(Z)^{(2q+1)}=I(P_0+P_1+P_2)^{(2q+1)}\\
&\cdot I((m_0-1)P_0+(m_1-1)P_1+(m_2-1)P_2)^{(2q+1)}=I(P_0+P_1+P_2)^{(2q)}\\
&\cdot I(P_0+P_1+P_2)\cdot I((m_0-1)P_0+(m_1-1)P_1+(m_2-1)P_2)^{2q+1}\\
&= I(Z) \cdot I(Z)^{(2q)} =I(Z)\cdot I(Z)^{(2i)}\cdot I(Z)^{(2q-2i)}\\
&=I(Z)^{(2i+1)}\cdot I(Z)^{(2q-2i)}
\end{flalign*} 
and the desired result follows.
\end{proof}
As a consequence of the results which were proved in \cite[Theorem 3.4]{bocci2010resurgence}, we can deduce the following corollary for three simple points in $\mathbb{P}^2$.
\begin{corollary} \label{cor4}
Let $P_0=[1:0:0]$, $P_1=[0:1:0]$, $P_2=[0:0:1]$. Then $$ \rho(I_{\mathbb{P}^2}(P_0+P_1+P_2))=4/3.$$	
\end{corollary}
From the previous corollary we can deduce the following useful lemma.
\begin{lemma} \label{lemma3}
Let $P_0, P_1$ and $P_2$ be noncollinear points in $\mathbb{P}^N$. Then $$ I(P_0+P_1+P_2)^{(r)} \subseteq I(P_0+P_1+P_2)^{r-1} \textrm{ for }  1\leq r \leq 4.$$
\end{lemma}
\begin{proof}
We work by induction on $r$. It is trivial for $r=1$ . For the induction suppose that it is true for $r-1$ and we prove it for $r$. Consider
$\mathcal{N}=x_0^{a_0}x_1^{a_1}\cdots x_N^{a_N} \in I(P_0+P_1+P_2)^{(r)} $
then
\begin{eqnarray} \label{condizio1}  
\begin{cases}
a_1+a_2+a_3+\cdots+a_N \geq r \\
a_0+a_2+a_3+\cdots+a_N \geq r \\ 
a_0+a_1+a_3+\cdots+a_N \geq r.
\end{cases} \end{eqnarray}
Set $b=\sum_{i=3}^N{a_i}$. We have the following cases:\\
\textbf{(a)} Assume $b=0$. We can see the monomial $\mathcal{N} $ as an element of the ideal
$I_{\mathbb{P}^2}(P_0+P_1+P_2)^{(r)}. $
By Corollary \ref{cor4},  $ \rho(I_{\mathbb{P}^2}(P_0+P_1+P_2))=4/3$. Furthermore, $r < 4$ implies $ 4r-4 < 3r $,  so $r/(r-1) > 4/3=\rho(I_{\mathbb{P}^2}(P_0+P_1+P_2))$. Then, using the definition of resurgence $ I_{\mathbb{P}^2}(P_0+P_1+P_2)^{(r)} \subseteq I_{\mathbb{P}^2}(P_0+P_1+P_2)^{r-1}$  for $r<4$, while it is possible to check computationally that $ I_{\mathbb{P}^2}(P_0+P_1+P_2)^{(4)} \subseteq I_{\mathbb{P}^2}(P_0+P_1+P_2)^{3}$. Thus $ \mathcal{N} \in I_{\mathbb{P}^2}(P_0+P_1+P_2)^{r-1}  $. Hence, $ \mathcal{N} \in I(P_0+P_1+P_2)^{r-1}$.\\
\textbf{(b)} Assume $ \sum_{i=3}^N{a_i}=b>0$.  There exists $i \in \left\{ 3,\ldots,N\right\}$ such that $a_i>0$.
We  may assume $i= \nolinebreak 3$.
We  write $ \mathcal{N}= \nolinebreak(x_3)\cdot (x_0^{a_0}x_1^{a_1}x_2^{a_2}x_3^{a_3-1}\cdots x_{N}^{a_N})$, where $x_3 \in I(P_0+P_1+P_2)$.
By (\ref{condizio1}) it follows that $ x_0^{a_0}x_1^{a_1}x_2^{a_2}x_3^{a_3-1}\cdots x_{N}^{a_N} \in I(P_0+P_1+P_2)^{(r-1)} \subseteq I(P_0+P_1+P_2)^{r-2}$, where the last inclusion holds for the induction.
 Hence, $\mathcal{N} \in I(P_0+P_1+P_2)^{r-1}$.
\end{proof}
Now we can prove the following important lemma.
\begin{lemma} \label{theo4}
Let $P_0, P_1$ and $P_2$ be noncollinear points in $\mathbb{P}^N$ and $m_2 \geq \max(m_0,m_1)$ and suppose that $m_0+m_1 > m_2 $ and $m_0+m_1+m_2$  is odd. Let $Z=m_0P_0+m_1P_1+m_2P_2$ be a scheme of fat points.  Then
\begin{itemize}	
\item[\textbf{(a)}] $I(Z)^{(q(1+\sum{m_i}))} \subseteq I(Z)^{q(\sum{m_i})}$ for all $q \in \mathbb{N}$, 
\item[\textbf{(b)}]
$	I(Z)^{(q(1+\sum{m_i})+r)}\subseteq I(Z)^{q(\sum{m_i})+r-1}$  for all $q \in \mathbb{N}$ and $0<r<1+\sum{m_i}$.
	\end{itemize}
\end{lemma}
\begin{proof}
Let us start with proving \textbf{(a)} by induction on $q$. First, we let $q=1$ as the base case. Thus, we need to prove  $ 	I(Z)^{(1+\sum{m_i})} \subseteq I(Z)^{\sum{m_i}}$.
Set $ Z_1=P_0+P_1+P_2 $ and $  Z_2=(m_0-1)P_0+(m_1-1)P_1+(m_2-1)P_2$, and we define:
$$ W(n_0,n_1,n_2)=(n_0+\sum_{i=0}^2{n_i})P_0+(n_1+\sum_{i=0}^2{n_i})P_1+(n_2+\sum_{i=0}^2{n_i})P_2,$$
for $n_i \geq \nolinebreak 1$. We claim that $I(W(n_0,n_1,n_2)) \subseteq I(Z_1)^{\sum{n_i}}$,  for all  $n_i \geq 1$.\\
\noindent \textbf{Proof of the claim}.
We prove by induction on the  sum $\sum_{i=0}^2{n_i}$. The base case is $\sum_{i=0}^2{n_i}=3$, with  $n_0=n_1=n_2=1$ and by  Lemma \ref{lemma3} it holds.  Now, we suppose the claim  holds for  $n
_i'$ such that $\sum_{i=0}^2{n'_i}<\sum_{i=0}^2{n_i}$ and we prove it for $n_i$.
Because we have already considered the case  $n_0=n_1=n_2=1$, there must exist an $i$ such that $n_i>1$. We can assume that $n_2>1$.
We consider the monomial $\mathcal{N}=x_0^{a_0}x_1^{a_1}\cdots x_N^{a_N} \in I(W(n_0,n_1,n_2))$.
Set $b=\sum_{i=3}^N{a_i}$.\\
\textbf{(i)} Let $ b=0$. We have the following subcases.
\begin{itemize}
\item[\textbf{(1)}]
Let $a_1=0$. By $\textrm{Cond}(W(n_0,n_1,n_2))$, it follows $ a_0\geq n_0+n_1+2n_2 \geq \sum_{i=0}^2{n_i}$ and $a_2\geq 2n_0+n_1+n_2 \geq \sum_{i=0}^2{n_i}$, then it can be written $\mathcal{N}=(x_0x_2)^{\sum{n_i}} x_0^{a_0-\sum{n_i}}x_2^{a_2-\sum{n_i}} \in I(Z_1)^{\sum{n_i}}$, because $x_0x_2 \in I(Z_1)$.
\item[\textbf{(2)}]
Let $a_0=0$: similar to the subcase $a_1=0$.
\item[\textbf{(3)}]
$ a_1,a_0 >0$: we can write, $\mathcal{N}= (x_0x_1)x_0^{a_0-1}x_1^{a_1-1} x_2^{a_2}$, where $x_0x_1 \in I(Z_1)$.\\
Using the fact that the $a_i$'s satisfy  $\textrm{Cond}(W(n_0,n_1,n_2))$, we can check that $x_0^{a_0-1}x_1^{a_1-1} x_2^{a_2} \in I(W(n_0,n_1,n_2-1)).$  So, by induction ($n_2-1 \geq 1$) 
$ x_0^{a_0-1}x_1^{a_1-1} x_2^{a_2} \in I(W(n_0,n_1,n_2-1)) \subseteq  I(Z_1)^{\sum{n_i}-1},$
and $\mathcal{N} \in I(Z_1)^{\sum{n_i}}. $
\end{itemize}
\textbf{(ii)} Let $ \sum_{i=3}^N{a_i}=b >0$. Without loss of generality, let $a_3 >0$. We can write
$ \mathcal{N}= (x_3) \cdot (x_0^{a_0}x_1^{a_1} x_2^{a_2}x_3^{a_3-1}\cdots x_N^{a_N})$,
 where $x_3 \in I(Z_1)$. By using $\textrm{Cond}(W(n_0,n_1,n_2))$, the second factor is in $  I(W(n_0,n_1,n_2-1))\subseteq I(Z_1)^{\sum{n_i}-1}$, where the last inclusion holds by induction. Hence $\mathcal{N} \in I(Z_1)^{\sum{n_i}}$. So, the claim is proved. Now, from the definition  $I(Z_1)^{1+\sum{m_i}} \cdot I(Z_2) \subseteq I(W(m_0,m_1,m_2)).$ By Corollary \ref{cor8}  it follows	that
\begin{flalign*}
I(Z)^{(1+\sum{m_i})}&=I(Z_1)^{(1+\sum{m_i})} \cdot I(Z_2)^{1+\sum{m_i}}\\
&=I(Z_1)^{(1+\sum{m_i})}\cdot I(Z_2) \cdot I(Z_2)^{\sum{m_i}}\\
&\subseteq  I(W(m_0,m_1,m_2)) \cdot I(Z_2)^{\sum{m_i}}\\
&\subseteq I(Z_1)^{\sum{m_i}}\cdot I(Z_2)^{\sum{m_i}}\\
&=(I(Z_1) \cdot I(Z_2))^{\sum{m_i}}=I(Z)^{\sum{m_i}},
\end{flalign*}
and the base case is proved.

 We suppose that \textbf{(a)} is true for  $q-1$, then we prove it for $q$.
By induction and Corollary \ref{cor9}, using the fact that $1+\sum{m_i}$ is even,
\begin{flalign*}
I(Z)^{(q(1+\sum{m_i}))}&=I(Z)^{((q-1)(1+\sum{m_i}))}\cdot I(Z)^{(1+\sum{m_i})} \\
&\subseteq I(Z)^{(q-1)(\sum{m_i})}\cdot I(Z)^{\sum{m_i}}=I(Z)^{q(\sum{m_i})}.
\end{flalign*}
For proving \textbf{(b)}, we work by induction on $q$ as before.
First of all, we need to prove the base case of $q=0$. Hence, we need to show $ 	I(Z)^{(r)} \subseteq I(Z)^{r-1} \textrm{ for } 1<r<1+\sum{m_i}$. Set $  Z_1=P_0+P_1+P_2 $ and $ Z_2=(m_0-1)P_0+(m_1-1)P_1+(m_2-1)P_2$ . We define:
$$ V(n_0,n_1,n_2,r)=(r+n_0-1)P_0+(r+n_1-1)P_1+(r+n_2-1)P_2,$$
for $n_i \geq 1$ and $1<r<1+\sum{n_i}$.	We claim that $I(V(n_0,n_1,n_2,r)) \subseteq I(Z_1)^{r-1}$ always holds.\\	
\noindent \textbf{Proof of the claim}.
We work by induction on the sum $\sum{n_i}$. The base case is   $n_i=1$. Then we have to prove 
$ I(P_0+P_1+P_2)^{(r)} \subseteq I(P_0+P_1+P_2)^{r-1}, \textrm{ for } 1<r<4$
and this is true by Lemma \ref{lemma3}.
We suppose that the claim is true for assignment $n
_i'$ such that $\sum{n_i'}<\sum{n_i}$, then we prove it for $n_i$.
Because we have already considered the case  $n_0=n_1=n_2=1$, there must exist an $i$ such that $n_i>1$. We can assume that $n_2>1$.
We consider $\mathcal{N}=x_0^{a_0}x_1^{a_1}\cdots x_N^{a_N} \in I(V(n_0,n_1,n_2,r))$.
Set $b=\sum_{i=3}^N{a_i}$, and we consider cases depending upon $b$.\\
\textbf{(i)}
let $ b=0$. We have the following subcases.
\begin{itemize}
\item[\textbf{(1)}]
$a_1=0$: by $\textrm{Cond}(V(n_0,n_1,n_2,r))$, it follows	that
$$ a_0\geq r+n_2-1\geq r-1	\textrm{  and } a_2\geq r+n_0-1 \geq r-1.$$			
So we can write	
$\mathcal{N}=(x_0x_2)^{r-1} x_0^{a_0-r+1}x_2^{a_2-r+1} \in I(Z_1)^{r-1}, $ because $x_0x_2 \in I(Z_1)$.
\item[\textbf{(2)}]
$ a_0=0$: similar to the case $a_1=0$.			
\item[\textbf{(3)}]
$a_1,a_0 >0$: we write,
	$ \mathcal{N}= (x_0x_1) x_0^{a_0-1}x_1^{a_1-1} x_2^{a_2},$ where $x_0x_1 \in I(Z_1)$.
 By  $\textrm{Cond}(V(n_0,n_1,n_2,r))$, we deduce $x_0^{a_0-1}x_1^{a_1-1} x_2^{a_2} \in I(V(n_0,n_1,n_2-1,r-1))$. 
So for the inductive step ($n_2-1 \geq 1$ and $r-1<(n_0+n_1+n_2-1)+1$) we conclude
$ x_0^{a_0-1}x_1^{a_1-1} x_2^{a_2} \in I(V(n_0,n_1,n_2-1,r-1)) \subseteq I(Z_1)^{r-2},$ and $\mathcal{N} \in I(Z_1)^{r-1}. $
\end{itemize}
\textbf{(ii)} Let $\sum_{i=3}^N{a_i}=b >0$. We can assume that $a_3 >0$. We can write
$ \mathcal{N}= (x_3) (x_0^{a_0}x_1^{a_1} x_2^{a_2}x_3^{a_3-1}\cdots x_N^{a_N}) $, where $x_3 \in I(Z_1)$. By $\textrm{Cond}(V(n_0,n_1,n_2,r))$, we have that\\
$$x_0^{a_0}x_1^{a_1} x_2^{a_2}x_3^{a_3-1}\cdots x_N^{a_N} \in I(V(n_0,n_1,n_2-1,r-1)).$$
So by induction ($n_2-1 \geq 1$ and $r-1<(n_0+n_1+n_2-1)+1$) we see\\
$ x_0^{a_0}x_1^{a_1} x_2^{a_2}x_3^{a_3-1}\cdots x_N^{a_N}  \in I(V(n_0,n_1,n_2-\nolinebreak1,r-1)) \subseteq I(Z_1)^{r-2}$, and $\mathcal{N} \in I(Z_1)^{r-1} $. So the claim is true. From the definition  $I(Z_1)^{r} \cdot I(Z_2) \subseteq I(V(m_0,m_1,m_2,r)).$ By Corollary \ref{cor8}, 
\begin{flalign*}
I(Z)^{(r)}&=I(Z_1)^{(r)}\cdot I(Z_2)^{r}=	I(Z_1)^{(r)} \cdot I(Z_2) \cdot I(Z_2)^{r-1}\\ 
&\subseteq I(V(m_0,m_1,m_2,r))\cdot I(Z_2)^{r-1} \subseteq I(Z_1)^{r-1}\cdot I(Z_2)^{r-1}\\
& = (I(Z_1)\cdot I(Z_2))^{r-1}=I(Z)^{r-1},
\end{flalign*}
and the base case is proved. Now we can proceed with the inductive step. We suppose that (b) is true for $q-1$, then we prove it for $q$. By induction and  Corollary \ref{cor9} we can write, using the fact that $1+\sum{m_i}$ is even,
\begin{flalign*}
	I(Z)^{(q(1+\sum{m_i})+r)}&=I(Z)^{((q-1)(1+\sum{m_i})+r)}\cdot I(Z)^{(1+\sum{m_i})}\\
     \subseteq &I(Z)^{(q-1)(\sum{m_i})+r-1}\cdot I(Z)^{\sum{m_i}}=I(Z)^{q(\sum{m_i})+r-1}.
\end{flalign*}
Thus the proof  is  complete.
\end{proof}
By Lemma \ref{theo4} we can deduce the following crucial corollary.
\begin{corollary} \label{cor5}
	Let $P_0, P_1$ and $P_2$ be noncollinear points in $\mathbb{P}^N$ and $m_2 \geq \max(m_0,m_1)$, and suppose that $ m_0+m_1 > m_2 $ and $\sum_{i=0}^2m_i $ is odd. If $Z=m_0P_0+m_1P_1+m_2P_2$, then
	$ \rho(I(Z)) \leq (1+\sum_{i=0}^2m_i)/(\sum_{i=0}^2m_i)$.
\end{corollary}
\begin{proof}
It is enough to show that if $m/n \geq (1+\sum{m_i})/(\sum{m_i}) $ then  $ I(Z)^{(m)} \subseteq \nolinebreak I(Z)^n$. Suppose that $m$ and $n$ are such that
$m/n \geq (1+\sum{m_i})/(\sum{m_i}) $. Then we can deduce $ m \geq \left\lceil \frac{1+\sum{m_i}}{\sum{m_i}}n \right\rceil$. Now, $n$ can be written as $n=q\sum{m_i}+r$ with $0\leq r<\sum{m_i}$. Thus 
$$ m \geq \left\lceil  q(1+\sum{m_i})+r+\frac{r}{\sum{m_i}}\right\rceil= \begin{cases}
q(1+\sum{m_i}) \textrm{ if } r = 0 \\
q(1+\sum{m_i})+r+1 \textrm{ if } r \neq 0.
\end{cases} $$
If $r=0$, by Lemma \ref{theo4} 
$$I(Z)^{(m)} \subseteq I(Z)^{(q(1+\sum{m_i}))} \subseteq I(Z)^{q\sum{m_i}}=I(Z)^n.$$
If $r \neq 0$, then $r'=r+1<1+\sum{m_i}$.   By Lemma \ref{theo4}
$$ I(Z)^{(m)} \subseteq I(Z)^{(q(1+\sum{m_i})+r')} \subseteq I(Z)^{q\sum{m_i}+r'-1}=I(Z)^n.$$
Then  $ \rho(I(Z)) \leq (1+\sum_{i=0}^2m_i)/ (\sum_{i=0}^2m_i)$.
\end{proof}
 From Corollary \ref{cor5} and Proposition \ref{prop} we can immediately deduce  Theorem \ref{casec}, therefore we have a complete description for the resurgence of a fat point scheme consisting of three noncollinear points of $\mathbb{P}^N$. 
\section*{Acknowledgement}
\thanks{This project started during the summer school PRAGMATIC 2017. The authors would like to thank Enrico Carlini,  Tai Huy Ha, Brian Harbourne and Adam Van Tuyl,  for giving very interesting lectures and for sharing open problems. They wish to thank all the organizers of PRAGMATIC 2017 for giving them the opportunity to attend the school. The authors owe also special thanks to the anonymous referee for the interesting and extensive comments on an earlier version of this paper.}

\end{document}